\documentclass[12pt,a4paper]{article}
\usepackage{amsmath,dsfont, amsthm, amsfonts, amssymb,tikz}
\usepackage{color,mathrsfs}
\usepackage{bbm}
\usepackage{amssymb}
\usepackage{txfonts}

\usepackage{picins}
\usepackage{wrapfig}

\usepackage[all]{xy}

\newlength{\defbaselineskip}
\theoremstyle{plain}
\newtheorem{theorem}{Theorem}[section]
\newtheorem{conjecture}[theorem]{Conjecture}
\newtheorem{corollary}[theorem]{Corollary}
\newtheorem{lemma}[theorem]{Lemma}
\newtheorem{prop}[theorem]{Proposition}

\theoremstyle{definition}
\newtheorem{definition}{Definition}[section]
\theoremstyle{remark}

\numberwithin{equation}{section}
\begin{document}
\title{On the inductive blockwise Alperin weight condition for the unipotent blocks of finite groups of Lie type $E_6$}
\author{Yucong Du, Pengcheng Li, Shuyang Zhao \\School of Mathematical Sciences, Peking University, \\Beijing 100871, China}
\date{}
\maketitle
\let\thefootnote\relax\footnotetext{Supported by NSFC (No. 11631001).}

\let\thefootnote\relax\footnotetext{\emph{Email addresses}: \leftline{duke@pku.edu.cn (Y. Du), pcli17@pku.edu.cn (P. Li), zsy0509@pku.edu.cn (S. Zhao).}}
\begin{abstract}
In this article, we consider the finite exceptional groups of Lie type $E_6$ and $^2E_6$. We prove the inductive blockwise Alperin weight condition holds for unipotent $l$-blocks of $E_6^{\varepsilon}(q)$ if $2,3\nmid q$, $l\geq 5$.
\end{abstract}

\noindent\textbf{2010 Mathematics Subject Classification:} 20C20, 20C33.

\noindent\textbf{Keywords:} Alperin weight conjecture; inductive blockwise Alperin weight condition; finite groups of Lie type $E_6$; unipotent blocks.
\section{Introduction}
Let $G$ be a finite group and $l$ a prime. For an $l$-subgroup $R$ of $G$ and a character
$\varphi\in\mathrm{Irr}(N_G(R))$, the pair $(R,\varphi)$ is called an $l$-weight of $G$ if $R\leq \ker(\varphi)$ and $\varphi$ is of $l$-defect zero when viewed as a character of $N_G(R)/R$. Let $B$ be an $l$-block of $G$. An $l$-weight $(R,\varphi)$ is called a $B$-weight if $\mathrm{bl}(\varphi)^G = B$, where $\mathrm{bl}(\varphi)$ means the $l$-block of $N_G(R)$ containing $\varphi$. Note that $R$ must be $l$-radical for any $l$-weight $(R,\varphi)$, that is, $O_l(N_G(R))=R.$ Denote by $\mathcal{W}(B)$ the set of all $G$-conjugacy classes of $B$-weights of $G$. Then the blockwise Alperin weight conjecture
can be stated as follows:
\begin{conjecture}(\cite{conj})
 Let $G$ be a finite group, $l$ a prime and $B$ an $l$-block of $G$. Then
$$|\mathcal{W}(B)|=|\mathrm{IBr}(B)|.$$
\end{conjecture}
This conjecture has been reduced to simple groups, that is, this conjecture holds for a finite group $G$ if all non-abelian simple groups involved in $G$ satisfy the so-called inductive blockwise Alperin weight (iBAW) condition. See Section \ref{IBAW} below.

By now, the (iBAW) condition has been verified for some cases. For example, cyclic blocks, see \cite{Cyc}; many of the sporadic simple groups, see \cite{spor}; simple alternating groups, simple groups of type $^2F_4$, see \cite{abe}; simple groups of Lie type in the defining characteristic, simple groups of type $^2B_2$ and $^2G_2$, see \cite{red}; simple groups of type $G_2$ and $^3D_4$, see \cite{GD}; some cases of type A, see \cite{PSL}, \cite{PSU} and \cite{blockA}; the case of type C under the assumption that the decomposition matrix with respect to $\mathcal{E}(G,l')$ is unitriangular, see \cite{C2} and \cite{Codd}; unipotent blocks of classical groups, see \cite{clas}.

In this article, we consider finite exceptional types of Lie type $E_6$ and $^2E_6$ based on the classification of their radical subgroups in \cite{RadE6}. We write $E_6^{+1}(q)=E_6(q)$ and $E_6^{-1}(q)=$$^2E_6(q)$ for the universal versions.

Our main theorem can be stated as follows:
\begin{theorem}
Let $G$ be the finite exceptional groups of Lie type $E_6^{\varepsilon}(q)$ of universal version, where $\varepsilon\in\{\pm 1\}$. Let $l\geq 5$ be a prime with $l\nmid q$. Let $B$ be a unipotent block of $G$. If the sub-matrix of the decomposition matrix of $B$ with respect to the basic set $\mathcal{E}(G,1)\cap B$ is unitriangular, then the inductive blockwise Alperin weight condition holds for $B$.

In particular, if $2,3\nmid q$, then the inductive blockwise Alperin weight condition holds for every unipotent $l$-block of $G$.
\end{theorem}
\section{Notations and preliminary results}
\subsection{Notations}
Let $G$ be a finite group and $l$ a prime. The set of all irreducible ordinary or $l$-Brauer characters of a finite group $G$ is denoted by $\mathrm{Irr}(G)$ or $\mathrm{IBr}(G)$ respectively. The set of all irreducible ordinary or $l$-Brauer characters belonging to a block $B$ is denoted by $\mathrm{Irr}(B)$ or $\mathrm{IBr}(B)$. For $\chi\in\mathrm{Irr}(G)\cup\mathrm{IBr}(G)$, we denote by $\mathrm{bl}(\chi)$ the block that $\chi$ belongs to. The set of $l$-blocks of $G$ is denoted by $\mathrm{bl}(G).$ We denote by $\mathrm{dz}(G)$ the set of $l$-defect zero irreducible
characters of $G$. If $Q$ is a radical $l$-subgroup of $G$ and $B$ an $l$-block of $G$, then we define the set
$\mathrm{dz}(N_G(Q)/Q, B) := \{\chi\in\mathrm{dz}(N_G(Q)/Q)\mid\mathrm{bl}(\chi)^G = B\}$. If $H$ is a subgroup of  $G$ and if $\chi$ and $\nu$ are characters of $G$ and $H$ respectively, then we denote by $\chi_H$ and $\nu^G$ the restriction of $\chi$ to $H$ and the induction of $\nu$ to $G$ respectively. By $\mathrm{Irr}(G\mid\nu)$ we denote the set of irreducible constituents of $\nu^G$. If $H$ is normal in $G$ we write $\mathrm{Irr}(H\mid\chi)$ for the set of constituents of $\chi_H$. Let $\chi\in \mathrm{Irr}(G)$. We denote by $\chi^\circ$ the restriction of $\chi$ to the set of all $l'$-elements of $G$. Let $Y \subseteq \mathrm{IBr}(G)$. A subset $X\subseteq \mathrm{Irr}(G)$ is called a basic set of $Y$ if $\{\chi^\circ\mid\chi\in X\}$ is a $\mathbb{Z}$-basis of $\mathbb{Z}Y$. If $Y=\mathrm{IBr}(B)$ for some $l$-block $B$ of $G$, then we also say $X$ a basic set of $B$.

If a group $A$ acts on a finite set $X$, we denote by $A_x$ the stabilizer of $x\in X$ in $A$, analogously we denote by $A_{X'}$ the setwise stabilizer of $X'\subseteq X$. If $A$ acts on a finite group $G$ by automorphisms, then there is a natural action of $A$ on $\mathrm{Irr}(G)\cup \mathrm{IBr}(G)$ given by $^{a^{-1}}\chi(g) = \chi^a(g) = \chi(g^{a^{-1}})$ for every $g\in G, a \in A $ and $\chi\in\mathrm{Irr}(G)\cup\mathrm{IBr}(G)$.

A subgroup $R\leq G$ is called $l$-radical if $R = O_l(N_G(R))$. We denote by $\mathrm{Rad}(G)$ the set of $l$-radical subgroups of G. Furthermore, $\mathrm{Rad}(G)/ \sim_G$ denotes a $G$-transversal of radical $l$-subgroup of $G$.

Following the notation in \cite{RadE6}, for finite groups $A,B$, we write $A.B$ for an extension of $A$ by $B$. If $n,m$ are positive integers, we write $n^m$ for the direct product of $m$ copies of cyclic groups of order $n$. For $\eta\in\{\pm 1\}$, we define $n_\eta=(n,q-\eta)$. If $H_1,H_2\leq G$ and $Z\leq Z(H_1)\cap Z(H_2)$, we denote by $H_1\circ_Z H_2$ the central extension of $H_1$ and $H_2$ over $Z$; we also write $H_1\circ H_2=H_1\circ_{Z(H_1)\cap Z(H_2)}H_2.$
\subsection{The inductive blockwise Alperin weight condition}\label{IBAW}
There are several versions of the iBAW condition. Apart from the original version given in
\cite[Defition 4.1]{red}, there is also a version treating only blocks with defect groups involved in certain
sets of $l$-groups \cite[Defition 5.17]{red}, or a version handling single blocks in \cite[Definition 3.2]{Cyc}. We shall
consider the inductive condition for a single block here.
\begin{definition}(\cite[Definition 3.2]{Cyc})\label{ibaw}
Let $l$ be a prime, $S$ a finite non-abelian simple group and $X$
the universal $l'$-covering group of $S$. Let $B$ be an $l$-block of $X$. We say the inductive blockwise
Alperin weight (iBAW) condition holds for $B$ if the following statements hold:
\begin{enumerate}
  \item[(1)] There exist subsets $\mathrm{IBr}(B\mid R) \subseteq \mathrm{IBr}(B)$ for every $R \in \mathrm{Rad}(X)$ with the following properties:
\begin{enumerate}
\item $\mathrm{IBr}(B\mid R)^a = \mathrm{IBr}(B\mid R^a)$ for every $R \in \mathrm{Rad}(X)$, $a \in\mathrm{ Aut}(X)_B$,
\item $\mathrm{IBr}(B)=\dot{\bigcup}_{R\in\mathrm{Rad}(X)/ \sim_X }\mathrm{IBr}(B\mid R)$.
\end{enumerate}
  \item[(2)] For every $R\in\mathrm{Rad}(X)$ there exists a bijection
  $$\Omega_R^X:\mathrm{IBr}(B\mid R)\rightarrow \mathrm{dz}(N_X(R)/R, B) $$
  such that $\Omega_R^X(\varphi)^a=\Omega_{R^a}^X(\varphi^a)$ for all $\varphi\in\mathrm{IBr}(B\mid R)$ and $a\in\mathrm{Aut}(X)_B.$
  \item[(3)] For every $R \in \mathrm{Rad}(X)$ and every $\varphi\in\mathrm{IBr}(B\mid R)$, there exist a finite group $A:=A(\varphi,R)$ and $\tilde{\varphi}\in\mathrm{IBr}(A)$, $\tilde{\varphi'}\in\mathrm{IBr}(N_A(\bar{R}))$ where we use the notation
      $$\bar{R}=RZ/Z\mathrm{\ and\ } Z=Z(X)\cap \ker(\varphi),$$
      with the following properties:
\begin{enumerate}
  \item for $\bar{X} := X/Z$ the group $A$ satisfies $\bar{X}\unlhd A$, $A/C_A(\bar{X})\cong\mathrm{Aut}(X)_\varphi$, $C_A(\bar{X}) = Z(A)$ and $l\nmid |Z(A)|$,
  \item $\tilde{\varphi}$ is an extension of the $l$-Brauer character of $\bar{X}$ associated with $\varphi$,
  \item  $\tilde{\varphi'}$ is an extension of the $l$-Brauer character of $N_{\bar{X}}(\bar{R})$ associated with the inflation of $\Omega_R^X(\varphi)^\circ\in\mathrm{IBr}(N_X(R)/R)$ to $N_X(R)$,
  \item $\mathrm{bl}(\tilde{\varphi}_J)=(\mathrm{bl}(\tilde{\varphi'}_{N_J(\bar{R})}))^J$ for every subgroup $J$ satisfying $\bar{X }\leq J \leq A$.
\end{enumerate}
\end{enumerate}
\end{definition}
\begin{definition}
Let $l$ be a prime, $S$ a finite non-abelian simple group and $X$ the universal $l'$-covering group of $S$. We say that the inductive blockwise Alperin weight (iBAW) condition holds for $S$ and the $l$ if the (iBAW) condition holds for every $l$-block of $X$.

\end{definition}
\begin{lemma}(\cite[Lemma 2.11]{GD})\label{equbij}
Let $l$ be a prime, $S$ a finite non-abelian simple group and $X$ the universal $l'$-covering group of $S$. Let $B$ be an $l$-block of $X$. If there is an $\mathrm{Aut}(X)_B$-equivariant bijection between $\mathrm{IBr}(B)$ and $\mathcal{W}(B)$, then there are natural defined sets $\mathrm{IBr}(B\mid Q)$ and bijections
$\Omega_Q^X$ such that condition (1) and (2) of Definition \ref{ibaw}  hold for $B$.
\end{lemma}
For convenience, we define the so-called ``unipotent weight'' for later use.
\begin{definition}
Let $G$ be a finite group. We call a weight of $G$ unipotent weight if it's a $B$-weight for some unipotent block $B$ of $G$.
\end{definition}
\subsection{Representations of finite groups of Lie type}
Let $q=p^f$ be a power of a prime $p$ and $l$ be a prime number different from $p$ from now on.

Let $\mathbb{F}_q$ be the finite field with $q$ elements. Suppose $\mathbf{G}$ is a connected reductive algebraic group over $\overline{\mathbb{F}_q}$ and $F:\mathbf{G}\rightarrow\mathbf{G}$ a Frobenius endomorphism endowing $\mathbf{G}$ with an $\mathbb{F}_q$-structure. The group of rational points $G=\mathbf{G}^F$ is finite. Let $\mathbf{G}^*$ be the dual group of $\mathbf{G}$ with the corresponding Frobenius endomorphism also denoted by $F$. When $\mathbf{L}$ is an $F$-invariant Levi subgroup for $\mathbf{G}$ and $\mathbf{P}$ is a parabolic subgroup of $\mathbf{G}$ containing $\mathbf{L}$ as a Levi complement, we denote by $\mathbf{R}_\mathbf{L\subseteq P}^\mathbf{G}$ the Delinge-Lusztig induction and $^*\mathbf{R}_\mathbf{L\subseteq P}^\mathbf{G}$ the Delinge-Lusztig restriction of $\mathbf{L\subseteq P}$. For an $F$-stable subgroup $\mathbf{H}$ of $\mathbf{G}$ and $\chi\in\mathrm{Irr}(\mathbf{H}^F)$, we write $N_{\mathbf{G}^F}(\mathbf{H},\chi)=\{n\in N_{\mathbf{G}^F}(\mathbf{H})\mid\chi^n=\chi\}$ and $W_{\mathbf{G}^F}(\mathbf{H},\chi)=N_{\mathbf{G}^F}(\mathbf{H},\chi)/\mathbf{H}^F$.

If $L$ is a Levi subgroup of a finite group of Lie type $G$ and $P$ is a parabolic subgroup of $G$ containing $L$ as a Levi complement, we denote by $R_{L\subseteq P}^G$ the Harish-Chandra induction and $^*R_{L\subseteq P}^G$ the Harsh-Chandra restriction for $L\subseteq P$. Note that the definition is independent of the choice of $P$ by \cite[Theorem 2.4]{LS}, so we are allowed to omit $P$ from the notation. Also note that for $F$-stable parabolic $\mathbf{P}$ of $\mathbf{G}$, the Delinge-Lusztig induction $\mathbf{R}_\mathbf{L\subseteq P}^\mathbf{G}$ reduces to the Harish-Chandra induction $R_{\mathbf{L}^F}^{\mathbf{G}^F}$.

Let $s$ be a semisimple element of $\mathbf{G}^{*F}$, the Lusztig series associated with $s$ is denoted by $\mathcal{E}(\mathbf{G}^F, s)$. For a semisimple $l'$-element $s$ of $\mathbf{G}^{*F}$, define $\mathcal{E}_l(\mathbf{G}^F, s)$ as the union of $\mathcal{E}(\mathbf{G}^F, t)$ such that $s=t_{l'}$. By \cite[2.2 th\'{e}or\`{e}me]{UnLu},  $\mathcal{E}_l(\mathbf{G}^F, s)$ is a union of $l$-blocks. Moreover, for each $l$-block $B$ such that  $\mathrm{Irr}(B)\subseteq\mathcal{E}_l(\mathbf{G}^F, s)$, we have $\mathrm{Irr}(B)\cap \mathcal{E}(\mathbf{G}^F, s)\neq\varnothing.$ We denote by $\mathcal{E}(\mathbf{L}^F,l')$ the union of $\mathcal{E}(\mathbf{L}^F,s)$ for semisimple $l'$-elements $s$.

\begin{theorem}(\cite[Theorem A]{basicset})\label{basicset}
Let $l$ be a prime good for $\mathbf{G}$ and not dividing the defining characteristic of $\mathbf{G}$.
Assume that $l$ does not divide $|(Z(\mathbf{G})/Z^\circ(\mathbf{G}))^F|$. Let $s\in\mathbf{G}^{*F}$ be a semisimple $l'$-element. Then $\mathcal{E}(\mathbf{G}^F, s)$ forms a basic set of $\mathcal{E}_l(\mathbf{G}^F, s)$.
\end{theorem}

Following \cite{uniord} and \cite{CE2}, we use the terminology of Sylow $d$-theory. Let $d$ be a positive integer. An $F$-stable Levi subgroup $\mathbf{L}$ is called $d$-split if $\mathbf{L}=C_{\mathbf{G}}{(Z^\circ(\mathbf{L})_d)}$, where $Z^\circ(\mathbf{L})_d$ is the Sylow $d$-torus of $Z^\circ(\mathbf{L})$. Let $\mathbf{L}_i$ $(i=1,2)$ be a Levi subgroup and $\zeta_i$ $(i=1,2)$ a character of $\mathbf{L}_i^F$. We write $(\mathbf{L}_1,\zeta_1)\geq (\mathbf{L}_2,\zeta_2)$ if $\mathbf{L}_1\supseteq\mathbf{L}_2$ and there is a parabolic subgroup $\mathbf{P}_2\subseteq\mathbf{L}_1$ containing $\mathbf{L}_2$ as a Levi complement and $\langle\zeta_1,\mathbf{R}_{\mathbf{L}_2\subseteq\mathbf{P}_2}^{\mathbf{L}_1}\zeta_2\rangle\neq 0$.

One calls $\geq_d$ the same relation restricted to pairs $\mathbf{L}_1$, $\mathbf{L}_2$ of $d$-split Levi subgroups. We define $\gg_d$ as the transitive closure of $\geq_d.$ If $(\mathbf{L},\chi)$ is $\gg_d$-minimal, then $\chi\in\mathrm{Irr}(\mathbf{L}^F)$ is called a $d$-cuspidal character and $(\mathbf{L},\chi)$ is called a $d$-cuspidal pair. Moreover, such a pair is called a unipotent $d$-cuspidal pair if $\chi$ is a unipotent character.

Let $e$ be the multiplicative order of $q$ mod $l$ throughout out this article.

In \cite{uniord}, the author gives a label for unipotent blocks with unipotent $e$-cuspidal pairs. This was generalized in \cite{CE2}.
\begin{theorem}(\cite[Theorem 2.5 and 4.1]{CE2})\label{uniblo}
Assume $l\geq 7$ if some component of $\mathbf{G}$ is of type $E_8$ and $l\geq 5$ otherwise. Let $e$ be the multiplicative order of $q$ mod $l$. Then
\begin{enumerate}
\item[(1)] for any $e$-cuspidal pair $(\mathbf{L},\zeta)$ of $\mathbf{G}^F$ such that $\zeta\in\mathcal{E}(\mathbf{L}^F,l')$, there exists a unique block  $b_{\mathbf{G}^F}(\mathbf{L},\zeta)$ of $\mathbf{G}^F$ such that for all parabolic subgroup $\mathbf{P}$ containing $\mathbf{L}$ as a Levi complement, all irreducible constitutes of $\mathbf{R}_{\mathbf{L}\subseteq \mathbf{P}}^{\mathbf{G}}\zeta$ are  in $b_{\mathbf{G}^F}(\mathbf{L},\zeta)$;
\item[(2)] the map $$(\mathbf{L},\zeta)\mapsto b_{\mathbf{G}^F}(\mathbf{L},\zeta)$$ is a bijection between the $\mathbf{G}^F$-conjugacy classes of $e$-cuspidal pairs satisfying $\zeta\in\mathcal{E}(\mathbf{L}^F,l')$ and blocks of $\mathbf{G}^F$;
\item[(3)] $\mathrm{Irr}(b_{\mathbf{G}^F}(\mathbf{L},\zeta))\cap \mathcal{E}(\mathbf{G}^F,l')=\{\chi\in\mathrm{Irr}(\mathbf{G}^F)\mid(\mathbf{G},\chi)\gg_e(\mathbf{L},\zeta))\}$.
\end{enumerate}
In particular, $b_{\mathbf{G}^F}(\mathbf{L},\zeta)$ is a unipotent block if and only if $\zeta$ is a unipotent character. Moreover, in this case, $\geq_e$ and $\gg_e$ coincide.
\end{theorem}

\begin{theorem}(\cite[p.43 3.2]{gen})\label{weyl}
Let $(\mathbf{L},\zeta)$ be a unipotent $d$-cuspidal pair. Then there is a bijection between $\mathrm{Irr}(W_{\mathbf{G}^F}(\mathbf{L},\zeta))$ and $\mathrm{Irr}(\mathbf{G}^F\mid\mathbf{R}^\mathbf{G}_\mathbf{L}(\zeta))$.
\end{theorem}

Using the above two theorems we can prove the following proposition.

\begin{prop}\label{unicor}
Let $l$ be a prime that does not divide $|Z(\mathbf{G})/Z^\circ(\mathbf{G})|$. Assume $l\geq 7$ if some component of $\mathbf{G}$ is of type $E_8$ and $l\geq 5$ otherwise. Let $\mathbf{H}$ be an $e$-split levi of $\mathbf{G}$ and $b$ a block of $H=\mathbf{H}^F$. Then $b^{G}$ is a unipotent block of $G$ if and only if $b$ is a unipotent block of $H.$
\end{prop}
\begin{proof} By Theorem \ref{uniblo}, we may write $b=b_{H}(\mathbf{L},\lambda)$ for some $e$-cuspidal pair $(\mathbf{L},\lambda)$ of $\mathbf{H}$. Let $B=b_{G}(\mathbf{L},\lambda)$. Then $b$ is unipotent if and only if $B$ is unipotent. By \cite[Proposition 2.2]{uniord}, $\mathbf{L}=C_\mathbf{G}^\circ(Z(\mathbf{L})_l^F)$ and $\mathbf{L}^F=C_\mathbf{G}(Z(\mathbf{L})_l^F)^F$. It follows from \cite[Theorem 2.5]{CE2} that $\mathrm{bl}(\lambda)^H=b$ and $\mathrm{bl}(\lambda)^G=B$. In particular, $b^G=B$. This completes our proof.
\end{proof}

In most of our cases, the Sylow $l$-subgroup of $G$ is abelian. So we can apply the results in \cite{abe}, which we summarize as the following two theorems.
\begin{theorem}\label{levirad}
Let $l\geq 5$ be a prime. Suppose the Sylow $l$-subgroup of $G$ is abelian. Then
\begin{enumerate}
  \item[(1)]$N_{\mathbf{G}^F}(\mathbf{L})/\mathbf{L}^F$ has order prime to $l$ for any $e$-split Levi subgroup $\mathbf{L}$ of $\mathbf{G}$;
  \item[(2)] the two maps
$$\Phi: R\mapsto C_\mathbf{G}(R),\ \ \ \ \ \ \ \ \Psi: \mathbf{L}\mapsto Z(\mathbf{L})_l^F,$$
are mutually inverse $G$-equivariant bijections between the set of radical subgroups of $G$ and the set of $e$-split Levi subgroups of $G$;
\item[(3)]if $\mathbf{L}$ and $R$ are correspondent to each other as in (2), then $N_G(R)=N_G(\mathbf{L})$ .
\end{enumerate}
\end{theorem}
\begin{theorem}\label{malle}
Let $l\geq 5$ be a prime. Suppose the Sylow $l$-subgroup of $G$ is abelian. Let $(\mathbf{L},\zeta)$ be an $e$-cuspidal pair of $\mathbf{G}^F$ with $\zeta\in\mathcal{E}(\mathbf{L}^F,l')$. Let $R=Z(\mathbf{L})_l^F$. Then
\begin{enumerate}
  \item[(1)] $\delta(b_{G}(\mathbf{L},\zeta))=R$;
  \item[(2)] Let $Q$ be a radical subgroup of $G$ and $\lambda\in\mathrm{dz}(C_{G}(Q)/Q)$. Then $\mathrm{bl}(\lambda)^G=b_{G}(\mathbf{L},\zeta)$ if and only if $Q=R$ and $\lambda=\zeta$.
\end{enumerate}
In particular, if $\zeta$ is extensible to $N_{\mathbf{G}^F}(\mathbf{L},\zeta)$, then the weights of $b_{G}(\mathbf{L},\zeta)$ are exactly
$$\{(R,\phi^{N_G(R)})\mid\phi\in \mathrm{Irr}(N_{\mathbf{G}^F}(\mathbf{L},\zeta)\mid\zeta)\},$$
hence $$|\mathcal{W}( b_{G}(\mathbf{L},\zeta))|= |\mathrm{Irr}( W_{\mathbf{G}^F}(\mathbf{L},\zeta)|.$$
\end{theorem}

The unitriangular shape of the decomposition matrix plays a vital role in our proof, so we state it here for later use.
\begin{theorem}(\cite[Theorem A]{unitr})\label{unitr}
Let $l$ be a prime very good for $\mathbf{G}$. Assume $p$ is good for $\mathbf{G}$. Let $B$ be a unipotent block of $G=\mathbf{G}^F$. Then the sub-matrix of the decomposition matrix of $B$ with respect to the basic set $\mathcal{E}(G,1)\cap B$ is unitriangular.
\end{theorem}

From now on, let $G=E_6^\varepsilon(q)$ of universal version with $\varepsilon\in\{\pm 1\}$ and $\mathbf{G}$ a simple-connected algebraic group with Frobenius map $F$ such that $\mathbf{G}^F=G$. Assume $2,3\nmid q$ and $l\geq 5$. Let $\nu$ be the $l$-valuation such that $\nu(l)=1.$


\subsection{Structure of Levi subgroups}
In this subsection, we recall some facts on the structure of Levi subgroups (See also \cite[p.93-94]{levie}).

Let $\mathbf{G}$ be the group of Lie type $E_6^\varepsilon$, $\Phi$ be the set of simple roots of the root system of $\mathbf{G}$, $\Delta$ the Dynkin diagram of $\mathbf{G}$ and $W$ the Weyl group of $\mathbf{G}$. Let $\mathbf{T}$ be the maximal torus of $\mathbf{G}$ with $\mathbf{T}^F=(q-\varepsilon)^6$ depending on  $G=E_6^\varepsilon(q)$. For a subset $J$ of $\Phi$, let $\mathbf{L}_J$ be the Levi subgroup of $\mathbf{G}$ corresponding to $J$ and $W_J$ the subgroup of $W$ generated by the reflections corresponding to $J$. Let $N_W(J)=\{w\in W\mid wJ=J\}\cong N_W(W_J)/W_J$, then there is a natural isomorphism $N_W(J)\cong N_\mathbf{G}(\mathbf{L}_J)/\mathbf{L}_J$.

It's well-known that the $G$-conjugacy classes of the Levi subgroups of $G$ are parameterized by pairs $(J,w)$, where $J$ ranges through the $W$-conjugacy classes representatives of subsets of $\Phi$ which forms a proper subdiagram of $\Delta$ and $w$ ranges through the conjugacy classes representatives of $N_W(J)$. Let $n$ be an element of $\mathbf{G}$ mapping on $w$. By Lang's Theorem, there is an element $g\in\mathbf{G}$ such that $n=g^{-1}F(g)$. Then the Levi subgroup of $G$  parameterized by $(J,w)$ is indeed $$(^g\mathbf{L}_J)^F=(\mathbf{L}_J)^{nF}=\{x\in\mathbf{L}_J\mid nF(x)n^{-1}=x\}.$$

In the notations of \cite{levie}, we have $|\mathbf{L}_J^{nF}|=|\mathbf{Z}_J^{nF}||\mathbf{S}_J^{nF}|$, where $\mathbf{S}_J =[\mathbf{L}_J,\mathbf{L}_J]$ is the semisimple component of $\mathbf{L}_J$ and $\mathbf{Z}_J=Z^\circ(\mathbf{L}_J)$. Note that $\mathbf{Z}_J\leq \mathbf{T}$. Since $\mathbf{G}$ is universal, thus by \cite[Proposition 2.6.2]{CFSG}, we have $\mathbf{L}_J^{nF}=\mathbf{T}^{nF}\mathbf{S}_J^{nF}$ and $\mathbf{S}_J^{nF}$ is universal. Moreover, $\mathbf{T}^{nF}$ normalizes each complement of $\mathbf{S}_J^{nF}$ and induces diagonal automorphism on it. Hence we may write $\mathbf{L}_J^{nF}=(\mathbf{Z}_J^{nF}\circ_{\mathbf{Z}_J^{nF}\cap Z(\mathbf{S}_J^{nF})}\mathbf{S}_J^{nF}).Q$, where $Q$ induces outer-diagonal automorphism on $\mathbf{S}_J^{nF}$. Since $|\mathbf{L}_J^{nF}|=|\mathbf{Z}_J^{nF}||\mathbf{S}_J^{nF}|$, it follows that $|Q|=|\mathbf{Z}_J^{nF}\cap Z(\mathbf{S}_J^{nF})|=|(\mathbf{Z}_J\cap Z(\mathbf{S}_J))^{nF}|.$
\begin{lemma}\label{NL}
Keep the notations as above, then $N_G(^g\mathbf{L}_J)/(^g\mathbf{L}_J)^F\cong C_{N_W(J)}(w).$
\end{lemma}
\begin{proof}
This is indeed \cite[Proposition 26.4(b)]{Alggp}.
\end{proof}
\section{Unipotent blocks of $E_6^\varepsilon(q)$ with $l\mid q-\varepsilon$}

Suppose $l\mid q-\varepsilon$ in this section.
\subsection{Unipotent blocks}
In this section, we give a description of the unipotent blocks of $G$.
\begin{prop}
Suppose $l\mid q-1$. Then there are 4 unipotent blocks for $E_6(q)$, 2 of them are defect zero, the other two has 25 or 3 unipotent characters respectively. The block with 25 unipotent characters is the principal block.
\end{prop}
\begin{proof}
Note that $e=1$ in this case. In particular, $e$-cuspidal characters are just cuspidal characters. Note that the unipotent characters of finite exceptional groups of Lie type were investigated in \cite{exc}; see also \cite[Appendix]{lusztig} and \cite[\S 13.9]{Carter}. We can see that $E_6(q)$ has 30 unipotent characters. 25 of them lie in the principal series, 3 of them arise from the cuspidal unipotent character of the Levi subgroup $D_4(q)$ and 2 are cuspidal. Hence the first statement  follows from Theorem \ref{uniblo}. The second statement follows easily from Brauer's Third Main Theorem (see \cite[\S 5 Theorem 6.1]{NT}). Finally, by an easy computation, we can see the two cuspidal unipotent characters in \cite[Appendix]{lusztig}  are defect zero.
\end{proof}
\begin{prop}\label{blo}
Suppose $l\mid q+1$. Then there are 4 unipotent blocks for $E_6^{-}(q)$, 2 of them are defect zero, the other two has 25,3 unipotent characters respectively. The block with 25 unipotent characters is the principal block.
\end{prop}
\begin{proof}
We have $e=2$ in this case. Since the maximal torus $\mathbf{T}$ with  $\mathbf{T}^F=(q+1)^6$ is a 2-split Levi subgroup of $E^-_6(q)$, we have $(\mathbf{T},1)$ is a 2-cuspidal pair, where 1 is the trivial character of $\mathbf{T}^F$. For this pair, we have $W_{\mathbf{G}^F}(\mathbf{T}, 1)=\mathrm{PSp}_4(3).2$ is the Weyl group of $E_6(q)$ with 25 irreducible characters by \cite[p.27]{Atlas}. So there are 25 unipotent characters in the unipotent block of $E^-_6(q)$ corresponding to $(\mathbf{T},1)$ by Theorem \ref{uniblo} and Theorem \ref{weyl}. By \cite[Table 1]{gen}, $$(\mathbf{L},\zeta)=((q+1)^2.\mathbf{D}_4(q),\phi_{13,02})$$ is another 2-split cuspidal pair with $W_{\mathbf{G}^F}(\mathbf{L},\zeta)=S_3$. So there are 3 unipotent characters in the unipotent block of $E^-_6(q)$ corresponding to such pair by a similar argument as the previous case. Moreover, there are two unipotent characters of defect zero by \cite[Appendix]{lusztig}. Since there are 30 unipotent characters in total, we have exhausted them. This completes our proof.
\end{proof}
In conclusion, there are two unipotent blocks of $G$ with positive defect. The principal block, which we denote by $B_1$, has 25 unipotent characters; the other block has 3 unipotent characters and we denote it by $B_2$.
\subsection{Weights}

In this section, we determine the weights for unipotent blocks of $G$. Let $B$ be a block of $G$. By \cite[p.3]{WeiSym}, we can find the $B$-weights of $G$ as follows: Firstly, determine all radical subgroups of $G$ up to conjugacy. Secondly, for each radical subgroup $R$, find all blocks $b$ of $C_G(R)R$ with defect group $R$ such that $B=b^G$. Thirdly, let $\theta$ be the canonical character of $b$ and $N(\theta)$ the inertial group of $\theta$ in $N_G(R)$. For each $\psi$ in $\mathrm{Irr}(N(\theta)\mid\theta)$, compute $d(\psi)$ which is defined by $\psi(1)=d(\psi){\theta(1)}$. Then each $\psi$ such that $d(\psi)_r=|N(\theta):C_G(R)R|_r$ gives rise to a $B$-weight $(R,\psi^{N_G(R)})$ of $G$.

The non-trivial radical $l$-subgroups of $G$ and their local structure up to conjugation have been classified in \cite[Theorem A]{RadE6}, see \cite[Table 2 and 4]{RadE6}. When $l=5$, they are denoted by $R_1,\cdots,R_{22};$ when $l\geq 7$, they are denoted by $R_1,\cdots,R_{16}.$ Note that $R_1,\cdots,R_{18}$ are abelian.

\begin{lemma}
For any $1\leq i\leq 16$, $C_\mathbf{G}(R_i)$ is an $e$-split Levi subgroup of $\mathbf{G}.$
\end{lemma}
\begin{proof}

Note that $$|G|=q^{36}(q^2-1)(q^5-\varepsilon)(q^6-1)(q^8-1)(q^9-\varepsilon)(q^{12}-1).$$
If $l\geq 7$, this is a direct consequence of \cite[Proposition 2.3]{abe}. For $l=5$, note that $R_i$ has the same structure as the case $l\geq 7$ and $R_i\leq (q-\varepsilon)^6$. So the proof of  \cite[Proposition 2.3]{abe} still holds.
\end{proof}
With Proposition \ref{unicor}, we see that in order to classify unipotent weights afforded by $R_1,\cdots,R_{16}$, it suffices to consider unipotent blocks of $C_G(R_1),\cdots,C_G(R_{16}).$ By \cite[Table 2]{RadE6}, for $1\leq i\leq 16$, $C_G(R_i)/R_i$ has the form $(A\circ H).Q$, where $A$ is an abelian $l'$-group, $H$ is a finite group of Lie type and $Q$ acts on it via diagonal automorphism. Hence in order to classify the defect zero unipotent blocks of $C_G(R_i)/R_i$, it suffices to classify the defect zero unipotent blocks of $H.Q$. Since every unipotent block contains a unipotent character, it suffices to classify the defect zero unipotent characters of $H.Q$.

According to \cite[Proposition 7.10]{DL}, the unipotent characters of a finite group of Lie type can be viewed as the unipotent characters of its adjoint version. Note that the unipotent characters of classical groups are parameterized by partitions or Lusztig symbols and their degrees are known; see \cite[Appendix]{lusztig} or \cite[\S 13.8]{Carter}, which are originally from \cite[Theorem 8.2]{unicla}, \cite{unitary} and \cite{GL}.

By \cite[Table 2]{RadE6}, $$\nu(|H.Q|)=\nu(|\mathrm{Spin}_{2n}^{\eta}(q)|)\mbox{ or } \nu(|\mathrm{SL}_{m}^{\eta}(q)|)$$ for some integer $m,n$ and $\eta\in\{\pm 1\}$. Thus we only need to consider unipotent defect zero characters of these groups.

\begin{lemma}\label{Sl}
Let $n\geq 2$ be a positive integer, $l$ a prime with $l\mid q-\varepsilon$ and $l\geq n$. Then $\mathrm{SL}_n^\varepsilon(q)$ has no defect zero unipotent characters.
\end{lemma}
\begin{proof}
Let $\chi^\alpha$ be the unipotent character of $\mathrm{SL}_n^\varepsilon(q)$ corresponding to the partition $\alpha=(\alpha_1,\cdots,\alpha_m)$ of $n$. Let $\lambda_i=\alpha_i+i-1$. By \cite[Appendix]{lusztig}, $\chi^\alpha$ has degree
$$\frac{(q-\varepsilon)\cdots(q^n-\varepsilon^n)\prod_{i'<i}(q^{\lambda_i}-\varepsilon^{i+i'}q^{\lambda_{i'}})}
{q^{\binom{m-1}{2}+\binom{m-2}{2}+\cdots}\prod_i\prod_{k=1}^{\lambda_i}(q^k-\varepsilon^k)}.$$
Thus
$$\nu(|\mathrm{SL}_n^\varepsilon(q)|)-\nu(\chi^\alpha(1))=\nu\left(\prod_i\prod_{k=1}^{\lambda_i}(q^k-\varepsilon^k)\right)-
\nu\left(\prod_{i'<i}(q^{\lambda_i-\lambda_{i'}}-\varepsilon^{i+i'})\right)-\nu(q-\varepsilon).$$
Note that $0\leq\lambda_i-\lambda_{i'}<l$ for all $i'<i$. Hence $$\nu(|\mathrm{SL}_n^\varepsilon(q)|)-\nu(\chi^\alpha(1))\geq\left(\sum_{i}\lambda_i-\binom{m}{2}-1\right)\nu(q-\varepsilon)=(n-1)\nu(q-\varepsilon)>0.$$
This completes the proof.
\end{proof}
\begin{lemma}\label{Spin}
Let $n$ be a positive integer, $l$ an odd prime with $l\mid q-1$ and $l\geq n+1$. Then $\mathrm{Spin}_{2n}^{+}(q)$ has no defect zero unipotent characters unless $n$ is an even square. Moreover, if $n$ is an even square, then $\mathrm{Spin}_{2n}^{+}(q)$  has a unique defect zero unipotent character.
\end{lemma}
\begin{proof}
Let $\chi^\alpha$ be the unipotent character of $\mathrm{Spin}_{2n}^{+}(q)$ corresponding to the Lusztig symbol $$\alpha=\binom{\lambda_1,\cdots,\lambda_a}{\mu_1,\cdots,\mu_b}$$ with $0\leq \lambda_1<\lambda_2<\cdots<\lambda_a,0\leq\mu_1<\mu_2<\cdots<\mu_b$, $\lambda_1,\mu_1$ are not both zero, $4\mid a-b$ and $$\sum\lambda_i+\sum\mu_j-\left[\left(\frac{a+b-1}{2}\right)^2\right]=n.$$ By \cite[Appendix]{lusztig}, $\chi^\alpha$ has degree
$$\frac{(q^2-1)\cdots(q^{2n-2}-1)(q^n-1)\prod_{i'<i}(q^{\lambda_i}-q^{\lambda_{i'}})\prod_{j'<j}(q^{\mu_j}-q^{\mu_{j'}})\prod_{i,j}(q^{\lambda_i}+q^{\mu_{j}})}
{2^cq^{\binom{a+b-2}{2}+\binom{a+b-4}{2}+\cdots}\prod_i\prod_{k=1}^{\lambda_i}(q^{2k}-1)\prod_j\prod_{k=1}^{\mu_j}(q^{2k}-1)}.$$
An easy computation shows $\mu_j,\lambda_i\leq n<l$, so $$\nu(|\mathrm{Spin}_{2n}^{+}(q)|)-\nu(\chi^\alpha(1))=\left(\sum_{i}\lambda_i+\sum_{j}\mu_j-\binom{a}{2}-\binom{b}{2}\right)\nu(q-\varepsilon).$$
Since $$\binom{a}{2}+\binom{b}{2}=\left[\left(\frac{a+b-1}{2}\right)^2\right]+\left[\left(\frac{a-b}{2}\right)^2\right],$$
it follows that
$$\nu(|\mathrm{Spin}_{2n}^{+}(q)|)-\nu(\chi^\alpha(1))=\left(n-\left(\frac{a-b}{2}\right)^2\right)\nu(q-\varepsilon).$$
Thus $\chi^\alpha$ can't be defect zero unless $$n=\left(\frac{a-b}{2}\right)^2.$$ If this holds, then $$\binom{a}{2}+\binom{b}{2}=\sum\lambda_i+\sum\mu_j.$$
Since $\lambda_1,\mu_1$ are not both zero, we may assume $b=0$, $a^2=4n$ and $\lambda_i=i$. Recall that $4\mid a-b$, hence the result follows.
\end{proof}
\begin{prop}
For $i\neq 6,16$, $C_G(R_i)$ has no unipotent block with defect group $R_i$. For $i=6,16$, $C_G(R_i)$ has a unique unipotent block with defect group $R_i$.
\end{prop}
\begin{proof}
If $\varepsilon=1$, then the result follows from Lemma \ref{Sl} and Lemma \ref{Spin}.

Now suppose $\varepsilon=-1$. According to Lemma \ref{Sl}, $C_G(R_i)$ has no unipotent block with defect group $R_i$ for $i\neq 1,6,16$. The result for $R_{16}$ is obvious.
Note that $$C_G(R_1)=((q-\varepsilon)\circ_{2_\varepsilon}\mathrm{Spin}_{10}^\varepsilon(q)).2_{\varepsilon}.$$ A direct computation shows that $\mathrm{Spin}_{10}^\varepsilon(q)$ has no defect zero unipotent block. For $$C_G(R_6)=((q-\varepsilon)^2\circ_{(2_\varepsilon)^2}(\mathrm{Spin}_{8}^+(q)).(2_{\varepsilon})^2,$$ by a direct computation, we can see $\mathrm{Spin}_{8}^+(q)$  has a unique defect zero unipotent block, which is labeled by $$\binom{1,3}{0,2}.$$ The character in this block has degree $$\frac{1}{2}q^3(q+1)^3(q^3+1).$$
\end{proof}
Let $b_1$ be the unique unipotent block of $C_G(R_6)$ with defect group $R_6$. Let $b_2$ be the unique unipotent block of $C_G(R_{16})$ with defect group $R_{16}$. Let $\theta_1,\theta_2$ be the canonical character of $b_1$ and $b_2$ respectively. Obviously, $b_2$ is the principal block and $\theta_2$ is the trivial character.

With the above preparations, we are now ready to compute the unipotent weight of $G$.
\begin{lemma}\label{s3}
Let $H$ be a finite group. Suppose $K$ is a normal subgroup of $H$ with $H/K\cong S_3$ and $\chi\in\mathrm{Irr}(K)$ is $H$-invariant. Then $\chi$ is extensible to $H$ and $|\mathrm{Irr}(H\mid\chi)|=3.$
\end{lemma}
\begin{proof}
Let $L$ be a normal subgroup of $H$ such that $L/K\cong\mathbb{Z}_3.$ Then by Clifford theory (see \cite[\S 3.5]{NT}), there is $\phi\in\mathrm{Irr}(L)$ which extends $\chi$. Moreover, $\mathrm{Irr}(L\mid\chi)=\{\phi,\phi\eta,\phi\eta^2\}$, where $\eta$ is a non-trivial irreducible character of $L/K.$ Note that $H/L$ has order 2 and $\mathrm{Irr}(L\mid\chi)$ is $H$-invariant. Thus by Clifford theory, there is at least one irreducible character of $\mathrm{Irr}(L\mid\chi)$ invariant under $H$. Without lost of generality, we may assume that $\phi$ is $H$-invariant. In particular, $\phi$ is extensible to $H$ and hence $\chi$ is extensible to $H$.  By Clifford theory again, we have $|\mathrm{Irr}(H\mid\chi)|=3.$
\end{proof}
\begin{prop}\label{wei7}
Assume $l\geq 7$ is a prime with $l\mid q-\varepsilon$. Then there are 28 unipotent weights for $G$ afforded by nontrivial radical $l$-group. More precisely, 25 of them are afforded by $R_{16}$, which are weights of $B_1$; the other 3 are afforded by $R_6$, which are weights of $B_2$.
\end{prop}
\begin{proof}
It suffices to consider $R_6$ and $R_{16}$. Since $C_G(R_6)$ contains a unique unipotent block with defect group $R_6$, $\theta_1$ must be $N_G(R_6)$-invariant. By \cite[Table 2]{RadE6}, $$N_G(R_{6})/C_G(R_{6})\cong S_3.$$ So it follows by Lemma \ref{s3} that  $|\mathrm{Irr}(N_G(R_6)\mid\theta)|=3.$ Since $l\nmid|S_3|$, there are 3 unipotent weights afforded by $R_6$. They are all $b_1^G$-weights. Since $b_1$ is not the principal block of $C_G(R_{6})$, $b_1^G$ can't be the principal block of $G$ by Brauer's Third Main Theorem.

As for $R_{16}$, we have$$N_G(R_{16})/C_G(R_{16})\cong  \mathrm{PSp}_4(3).2$$ by \cite[Table 2]{RadE6}. Recall that $\theta_2$ is the trivial character. Thus by Clifford theory, there are $|\mathrm{Irr}( \mathrm{PSp}_4(3).2)|$ extensions of $\theta_2$ to $N_G(R_{16})$. By \cite[p.27]{Atlas}, $ \mathrm{PSp}_4(3).2$ has $25$ irreducible characters. Note that $l\nmid |\mathrm{PSp}_4(3).2|$. It follows that there are 25  unipotent weights afforded by $R_{16}$. Since $b_2$ is the principal block of $C_G(R_{16})$, $b_2^G$ is the principal block of $G$ by Brauer's Third Main Theorem.
\end{proof}

\begin{corollary}
Assume $l=5$ with $l\mid q-\varepsilon$. Then there are 18 unipotent weights for $G$ afforded by $R_1,\cdots,R_{16}$. More precisely, 15 of them are afforded by $R_{16}$, which are weights of $B_1$; the other 3 are afforded by $R_6$, which are weights of $B_2$.
\end{corollary}
\begin{proof}
The proof of Proposition \ref{wei7} can almost be carried over. Note that although we can't use Lemma \ref{Spin} for $R_6$ now, we can get the same result by a direct computation. Also note that $ |\mathrm{PSp}_4(3).2|=2^7\cdot 3^4\cdot 5$ and it has 15 defect zero characters according to \cite[p.27]{Atlas}, so there are 15  unipotent weights afforded by $R_{16}$ in this case.
\end{proof}

\begin{prop}
Assume $l=5$ with $l\mid q-\varepsilon$. Then there is no unipotent weight afforded by $R_{17},R_{18},R_{20},R_{22}$, $8$ unipotent weights afforded by $R_{19}$ and $2$ unipotent weights afforded by $R_{21}$.
\end{prop}

\begin{proof}
For $i=17,20,22$, we have
$$
R_{i}C_G(R_{i})\leq C_G(R_3)=(\mathrm{GL}_5^\varepsilon(q)\circ_{2_\varepsilon}\mathrm{SL}_2(q)).2_\varepsilon,
$$
and for $i=18,19,21$, we have
$$
R_{i}C_G(R_{i})\leq C_G(R_5)=(q-\varepsilon)\times\mathrm{GL}_5^\varepsilon(q).
$$
Note that for $l\mid q-\varepsilon$, $\mathrm{GL}_n^\varepsilon(q)$ contains only one unipotent block, which is the principal block by \cite{FS}. Suppose that $b$ is a block of $R_iC_G(R_i)$ with defect group $R_i$ such that $b^G$ is a unipotent block in $G$. We have $b^{C_G(R_j)}$ is a unipotent block and hence the principal block of $C_G(R_j)$, where $j=3$ if $i=17,20,22$ and $j=5$ if $i=18,19,21$ respectively. By Brauer's third main theorem, we can see that $b^G$ is a unipotent block of $G$ if and only if $b$ is the principal block of $R_iC_G(R_i)$. In particular, $R_i\in\mathrm{Syl}_l(R_iC_G(R_i)).$ Thus $i=18,19,21.$ Note that
$$N_G(R_{18})/(C_G(R_{18})R_{18})= 10_\varepsilon,$$
$$
N_G(R_{19})/(C_G(R_{19})R_{19})=4\times2
$$
and
$$
N_G(R_{21})/(C_{G}(R_{21})R_{21})=\mathrm{SL}_2(5)\times 2
$$
as in \cite[Table 3]{RadE6} (we need to point out that the 2 is missing in the structure of $N_G(R_{21})$ in \cite[Table 3]{RadE6}). Since there is only one defect zero character of $\mathrm{SL}_2(5)$ by \cite[p.2]{Atlas}, we can see that there are $8$ unipotent weights afforded by $R_{19}$ and $2$ unipotent weights afforded by $R_{21}$ and others afford no unipotent weights.
\end{proof}

Combining the result in this section, we have the following result.
\begin{corollary}\label{wei5}
Assume $l=5$ with $l\mid q-\varepsilon$. Then there are 28 unipotent weights for $G$ afforded by nontrivial radical $5$-subgroups. More precisely, 15 of them are afforded by $R_{16}$, 8 of them are afforded by $R_{19}$, and 2 of them are afforded by $R_{21}$, which are weights of $B_1$; the other 3 are afforded by $R_6$, which are weights of $B_2$.
\end{corollary}
\subsection{Proof of the main theorem}
\begin{lemma}\label{equ}
Assume that the sub-matrix of the decomposition matrix of $B$ with respect to the basic set $\mathcal{E}(G,1)\cap B$ is unitriangular. Let $B$ be a unipotent block of $G$. Suppose $l\mid q-\varepsilon.$ Then there is an $\mathrm{Aut}(G)_{B}$-equivariant bijection between weights and Brauer characters of $B$.
\end{lemma}
\begin{proof}
Under our assumption, there is an $\mathrm{Aut}(G)_{B}$-equivariant bijection between $\mathrm{Irr}(B)\cap \mathcal{E}(G, 1)$ and $\mathrm{IBr}(B)$ by \cite[Lemma 7.5]{equ}. Thus it suffices to show there  is an $\mathrm{Aut}(G)_{B}$-equivariant bijection between weights and unipotent characters of $B$. Note that we have $|\mathrm{Irr}(B)\cap \mathcal{E}(G, 1)|=|\mathcal{W}(B)|$ by the results from the last two subsections. By \cite[Theorem 2.5]{uniext}, every unipotent character of $G$ is $\mathrm{Aut}(G)$-invariant. Hence it suffices to show every unipotent weight is $\mathrm{Aut}(G)$-invariant up to conjugation. Note that for a fixed $i$, the unipotent weights afforded by $R_i$ are the extensions of the same character of $R_iC_G(R_i)$.

For $i=6$, $N_G(R_i)/Z(C_G(R_i))\cong\mathrm{P\Omega}_8^+(q).S_4$. Thus the diagonal automorphism and the graph automorphism acts on these weights trivially. Note that these weights can be extended to $\mathrm{Aut}(\mathrm{P\Omega}_8^+(q))$. We know that the field automorphism acts on weights trivially.

For $i=16$, $N_G(R_i)/Z(C_G(R_i))\cong\mathrm{PSp}_4(3).2$, it has trivial outer-automorphism group, so the result follows.

For $l=5$ and $i=21$, the result is obvious since $N_G(R_i)/(R_iC_G(R_i))\cong\mathrm{SL}_2(5)\times 2$ and $\mathrm{SL}_2(5)$ has only one defect zero character.

So it remains to consider the case $l=5$ and $i=19$. Note that $N_G(R_{19})/T\leq N_G(T)/T\cong W$. Since the outer-automorphism group $G$ acts on $W$ trivially, we know that it acts on $N_G(R_{19})/T$ trivially and hence acts on $N_G(R_{19})/(R_{19}C_G(R_{19}))$ trivially. This completes the proof.
\end{proof}
\begin{proof}[Proof of Theorem 1.2 when $l\mid q-\varepsilon$:]
By Lemma \ref{equbij}, it suffices to check the condition (3) in Definition \ref{ibaw}.

Let $A=\mathrm{Aut}(G)$. Since every unipotent character of $G$ contains $Z(G)$ in its kernel, all irreducible Brauer characters also contain $Z(G)$ in their kernel by Theorem \ref{basicset}. In particular, $Z=Z(G)$. Moreover, every unipotent character of $G$ is extensible to $A$ when viewed as a unipotent character of $\bar{G}$  by \cite[Theorem 2.4 and 2.5]{uniext}, so any irreducible Brauer character $\varphi$ in $B$ is extensible to $\mathrm{Aut}(G)$ when viewed as a Brauer character of $\bar{G}$ by Theorem \ref{basicset} and \cite[Lemma 2.9]{clas}, which is denoted by $\tilde{\varphi}$. Thus we have shown (a) and (b) of condition (3).

Now we begin the proof of (c). It suffices to show that for $R\in\{R_6,R_{16},R_{19},R_{21}\}$ and $\varphi'\in\mathrm{dz}(N_{\bar{G}}(\bar{R})/\bar{R})$ a weight in a unipotent block, $\varphi'$ can be extended to some $\tilde{\varphi}'\in\mathrm{IBr}(N_A(\bar{R})/\bar{R})$. Denote $\mathrm{Out}(G)=\mathrm{Aut}(G)/\mathrm{Inn}(G).$

For $R=R_6$, we have $N_{\bar{G}}(\bar{R})\cong C_{\bar{G}}(\bar{R}).S_3$. Note that $\mathrm{ker}(\varphi')=Z(C_{\bar{G}}(\bar{R}))$. Since
$$
N_A(\bar{R})/\mathrm{ker}(\varphi')\cong\mathrm{P\Omega}_8^+(q).S_4.\mathrm{Out}(G)
$$
and $\mathrm{Aut}(\mathrm{P\Omega}_8^+(q))=\mathrm{P\Omega}_8^+(q).S_4.\langle F\rangle$, where $F$ is the field automorphism, we have
$$
N_A(\bar{R})/\mathrm{ker}(\varphi')\cong(\mathrm{P\Omega}_8^+(q).S_4\times \mathrm{Out}_0(G)).\langle F\rangle,
$$
where $\mathrm{Out}_0(G)$ is the subgroup of $\mathrm{Out}(G)$ generated by the diagonal automorphism and the graph automorphism. Thus $\varphi'$ can be extended to $N_A(\bar{R})$.

For $R=R_{16}$, we have $N_{\bar{G}}(\bar{R})\cong  C_{\bar{G}}(\bar{R}). W$ with $W\cong\mathrm{PSp}_4(3).2$. Similar to the previous case, we have
$$
N_A(\bar{R})/C_{\bar{G}}(\bar{R})\cong W\times\mathrm{Out}(G).
$$
Since $C_{\bar{G}}(\bar{R})\leq\mathrm{ker}(\varphi')$, $\varphi'$ can be extended to $N_A(\bar{R})$.

For $R=R_{19}$, we have $N_{\bar{G}}(\bar{R})\cong \bar{R}C_{\bar{G}}(\bar{R}).(4\times2)$. As is shown in Lemma \ref{equ}, the outer-automorphism group of $G$ acts on $N_G(R_{19})/(R_{19}C_G(R_{19}))$ trivially, so
$$
N_A(\bar{R})/\bar{R}C_{\bar{G}}(\bar{R})\cong4\times2\times\mathrm{Out}(G).
$$
Note that $\bar{R}C_{\bar{G}}(\bar{R})\leq\mathrm{ker}(\varphi')$. We know that $\varphi'$ can be extended to $N_A(\bar{R})$.

For $R=R_{21}$, we have $$N_G(R)\cong RC_G(R).(\mathrm{SL}_2(5)\times 2)\leq N_G(R_5)=((q-\varepsilon)^2.2\circ_5\mathrm{SL}_5^\varepsilon(q)).5$$ according to the proof of \cite[Theorem 3.6]{RadE6} and \cite[Lemma 3.5]{RadE6}. Note that the restriction of the diagonal automorphism, the field automorphism and the graph automorphism to $N_G(R_5)$ is exactly the diagonal automorphism, the field automorphism and the graph automorphism of $N_G(R_5)$ respectively. Since the outer-automorphism group of $\mathrm{SL}_2(5)$ has order 2, we have
$$
N_A(\bar{R})/\bar{R}C_{\bar{G}}(\bar{R})\cong (\mathrm{SL}_2(5)\times2\times\mathrm{Out}_0(G)).\langle F\rangle.
$$
Since $\bar{R}C_{\bar{G}}(\bar{R})\leq\mathrm{ker}(\varphi')$, we have $\varphi'$ can be extended to a character $\varphi''$ of $\bar{R}C_{\bar{G}}(\bar{R}).(\mathrm{SL}_2(5)\times2\times\mathrm{Out}_0(G))$ with $\mathrm{Out}_0(G)\leq\mathrm{ker}(\varphi'')$. Since the outer automorphism acts on $\varphi'$ trivially according to the proof of Lemma \ref{equ}, we know that $\varphi''$ can be extended to $N_A(\bar{R})$, which completes the proof.

For (d), first note that we can take $\tilde{\varphi}$ and $\tilde{\varphi'}$ such that $\mathrm{bl}(\tilde{\varphi}')^A=\mathrm{bl}(\tilde{\varphi})$.
Also note that $A=\bar{G}N_A(\bar{R})$, so we have $J=\bar{G}N_J(\bar{R})$. By \cite[Lemma 2.4]{ks}, for any $x\in N_J(\bar{R})$ we have
$$
\mathrm{bl}\left(\tilde{\varphi}'_{\langle N_{\bar{G}}(\bar{R}),x\rangle}\right)^{\langle\bar{G},x\rangle}= \mathrm{bl}\left(\tilde{\varphi}_{\langle\bar{G},x\rangle}\right).
$$
Hence
$$
\mathrm{bl}\left(\left(\tilde{\varphi}'_{N_J(\bar{R})}\right)_{\langle N_{\bar{G}}(\bar{R}),x\rangle}\right)^{\langle\bar{G},x\rangle}= \mathrm{bl}\left(\left(\tilde{\varphi}_J\right)_{\langle\bar{G},x\rangle}\right).
$$
Then by \cite[Lemma 2.5 (a)]{ks}, we have
$$
\mathrm{bl}\left(\tilde{\varphi}'_{N_J(\tilde{R})}\right)^J= \mathrm{bl}\left(\tilde{\varphi}_J\right).
$$
\end{proof}
\section{Unipotent blocks of $E_6^\varepsilon(q)$ with $l\nmid q-\varepsilon$}
Suppose $l\nmid q-\varepsilon$ in this section. Since $$|G|=q^{36}(q^2-1)(q^5-\varepsilon)(q^6-1)(q^8-1)(q^9-\varepsilon)(q^{12}-1),$$ every Sylow $l$-subgroup of $G$ is abelian.
\begin{prop}\label{nmid}
\begin{enumerate}
  \item[(1)] Assume $l\mid q+\varepsilon$. Then there are 5 unipotent blocks of $E_6^{\varepsilon}(q)$ , 3 of them are defect zero, the other 2 has 25,2 unipotent characters respectively. The block with 25 unipotent characters is the unipotent block.
  \item[(2)] Assume $l\mid q^2+\varepsilon q+1$. Then there are 5 unipotent blocks of $E_6^{\varepsilon}(q)$, 3 of them are defect zero, the other 2 has 24,3 unipotent characters respectively. The block with 24 unipotent characters is the unipotent block.
  \item[(3)] Assume $l\mid q^2+1$. Then there are 12 unipotent blocks of $E_6^{\varepsilon}(q)$, 10 of them are defect zero, the other 2 has 16,4 unipotent characters respectively. The block with 16 unipotent characters is the unipotent block.
  \item[(4)] Assume $l\mid q^2-\varepsilon q+1$. Then there are 10 unipotent blocks of $E_6^{\varepsilon}(q)$, 9 of them are defect zero, the other 1 has 21 unipotent characters. The block with 21 unipotent characters is the unipotent block.
\end{enumerate}
In all cases, the number of weights of each unipotent block equals to the number of unipotent characters in it.
\end{prop}
\begin{proof}
By \cite[Table 2]{tori}, $T_1=(q+\varepsilon)^2\times (q^2-1)^2$, $T_2=(q^2+\varepsilon q+1)^3$, $T_3=((q^2+1)(q-\varepsilon))^2$ and $T_4=(q^2-\varepsilon q+1)\times(q^4+q^2+1)$ are maximal tori of $G$. Moreover, $N_G(T_1)/T_1\cong 2_{+}^{1+4}.(S_3\times S_3)$ is the Weyl group of $F_4(q)$, $N_G(T_2)/T_2\cong U_3(2)$, $N_G(T_3)/T_3\cong U_2(3)$ and $N_G(T_4)/T_4\cong \mathrm{SL}_2(3)\times 3$ by \cite{NGT}. By \cite{f4}, $N_G(T_1)/T_1$ has 25 irreducible characters. Note that $N_G(T_2)/T_2,$ $N_G(T_3)/T_3$, $N_G(T_4)/T_4$ has 24, 16, 21 irreducible characters respectively. Thus (1)-(4) follow from \cite[Table 1]{gen} and  \cite[Appendix]{lusztig} as in the proof of Proposition \ref{blo}.

Let $(\mathbf{L},\zeta)$ be a unipotent $e$-cuspidal pair of $G$ with $\zeta\neq 1$. If $e\geq 2$, then $W_{\mathbf{G}^F}(\mathbf{L},\zeta)$ is cyclic by \cite[Table 1]{gen}. Otherwise we have $e=1$ and $\varepsilon=-1$. We see that $\mathbf{L}^F=\mathrm{GU}(6,q)$ and $\zeta$ is uniquely determined by \cite[\S 13.9]{Carter}. Hence $W_{\mathbf{G}^F}(\mathbf{L},\zeta)=N_{\mathbf{G}^F}(\mathbf{L})/\mathbf{L}^F\cong 2$ by Lemma \ref{NL} and \cite{levie}. In particular, the assumption of Theorem \ref{malle} always holds, so the result follows.
\end{proof}

\begin{proof}[Proof of Theorem 1.2 when $l\nmid q-\varepsilon$:]
We may assume the defect group of $B$ is not cyclic. Thus $l\mid q+1$ or $l\mid q^2+1$ or $l\mid q^2\pm q+1$ and $B$ is the principal block. By Proposition \ref{nmid} and Theorem \ref{basicset}, we have
$$
|\mathcal{W}(B)|= |\mathrm{IBr}(B)|.
$$
We denote $\bar{G}=G/Z(G)$ and $A=\mathrm{Aut}(G)$ as before. An argument as the previous section shows every irreducible Brauer character of $B$ is $\mathrm{Aut}(G)$-invariant and is extensible to $A$. Suppose $B$ is labeled by $(\mathbf{T},1)$ for some maximal torus $\mathbf{T}$ of $\mathbf{G}$, then $R=\mathbf{T}^F_l$. Take $\varphi'\in\mathrm{dz}(N_{\bar{G}}(\bar{R})/\bar{R})$. Then $\mathrm{ker}(\varphi')\geq C_{\bar{G}}(\bar{R})$. Note that $N_G(R)/C_G(R)=N_G(\mathbf{T}^F)/\mathbf{T}^F$ is a subgroup of $N_G(\mathbf{T}_0^F)/\mathbf{T}_0^F=\mathrm{PSp}_4(3).2$, where $\mathbf{T}_0^F=(q-\varepsilon)^6$. We have
$$
N_A(\bar{R})/C_{\bar{G}}(\bar{R})\cong N_G(R)/C_G(R)\times\mathrm{Out}(G).
$$
Hence all weights in the principal block are $\mathrm{Aut}(G)$-invariant up to conjugation and $\varphi'$ can be extended to $N_A(\bar{R})$. For (d), the proof for the case $l\mid q-\varepsilon$ still applies here.
\end{proof}

\bibliographystyle{plain}

\end{document}